\documentclass[11pt,twoside]{article}

\usepackage[a4paper,margin=3.5cm]{geometry}
\usepackage[utf8]{inputenc}
\usepackage[english]{babel}
\usepackage{csquotes}
\usepackage{amssymb}
\usepackage{amsmath}
\usepackage{amsthm}
\usepackage{bm}
\usepackage{dsfont}
\usepackage{enumitem}
\usepackage{epigraph}
\usepackage{fancyhdr}
\usepackage{mathrsfs}
\usepackage{mathtools}
\usepackage{graphicx}
\usepackage{array}
\usepackage{pgf,tikz}
\usepackage{tcolorbox}
\usepackage{tikz-cd}
\usepackage{titling}
\usepackage{url}
\usepackage{verbatim}
\usepackage[style=numeric-comp, backend=biber]{biblatex}
\usepackage{hyperref}
\hypersetup{
    colorlinks=true
}

\setlist{noitemsep}

\usetikzlibrary{arrows}
\usetikzlibrary{decorations.pathmorphing}
\usetikzlibrary{patterns}

\newcommand{\Bo}{\tn{Bo}}

\newcommand{\F}{\mathcal{F}}

\newcommand{\G}{\mathcal{G}}

\newcommand{\lra}{\longrightarrow}

\newcommand{\mc}{\mathcal}

\newcommand{\N}{\mathbb{N}}
\newcommand{\OO}{\mathcal{O}}

\renewcommand{\P}{\mathbb{P}}

\newcommand{\ra}{\rightarrow}

\newcommand{\stlra}[1]{\stackrel{#1}{\longrightarrow}}

\newcommand{\tn}{\mathrm}

\newcommand{\Z}{\mathbb{Z}}

\DeclareMathOperator{\cofib}{cofib}

\DeclareMathOperator*{\colim}{colim}

\DeclareMathOperator{\Cpl}{Cpl}

\DeclareMathOperator{\GL}{GL}

\DeclareMathOperator{\GW}{GW}

\DeclareMathOperator{\K}{K}

\DeclareMathOperator*{\limone}{lim^1} %lol, it's an Italian lemon

\DeclareMathOperator{\Perf}{Perf}

\DeclareMathOperator{\sheafhom}{\mathscr{H}\textit{\kern -4pt om}\,}

\DeclareMathOperator{\SO}{SO}
\DeclareMathOperator{\Spec}{Spec}

\DeclareMathOperator{\sPerf}{sPerf}

\DeclareMathOperator{\Vect}{Vect}
\DeclareMathOperator{\W}{W}

\numberwithin{equation}{subsection}

\theoremstyle{definition}
\newtheorem{definition}[equation]{Definition}

\theoremstyle{plain}
\newtheorem{corollary}[equation]{Corollary}
\newtheorem{lemma}[equation]{Lemma}
\newtheorem{proposition}[equation]{Proposition}
\newtheorem{theorem}[equation]{Theorem}
\newtheorem{maintheorem}{Theorem}

\title{\textsc{on atiyah-segal completion for t-equivariant hermitian k-theory}}
\author{
    \textsc{herman rohrbach}
    \thanks{
        The author was partially supported by the research training group \emph{GRK 2240: Algebro-Geometric Methods in Algebra, Arithmetic and Topology} and by the ERC through the project QUADAG.
        This paper is part of a project that has received funding from the European Research Council (ERC) under the European Union's Horizon 2020 research and innovation programme (grant agreement No. 832833). \newline
	\includegraphics[scale=0.08]{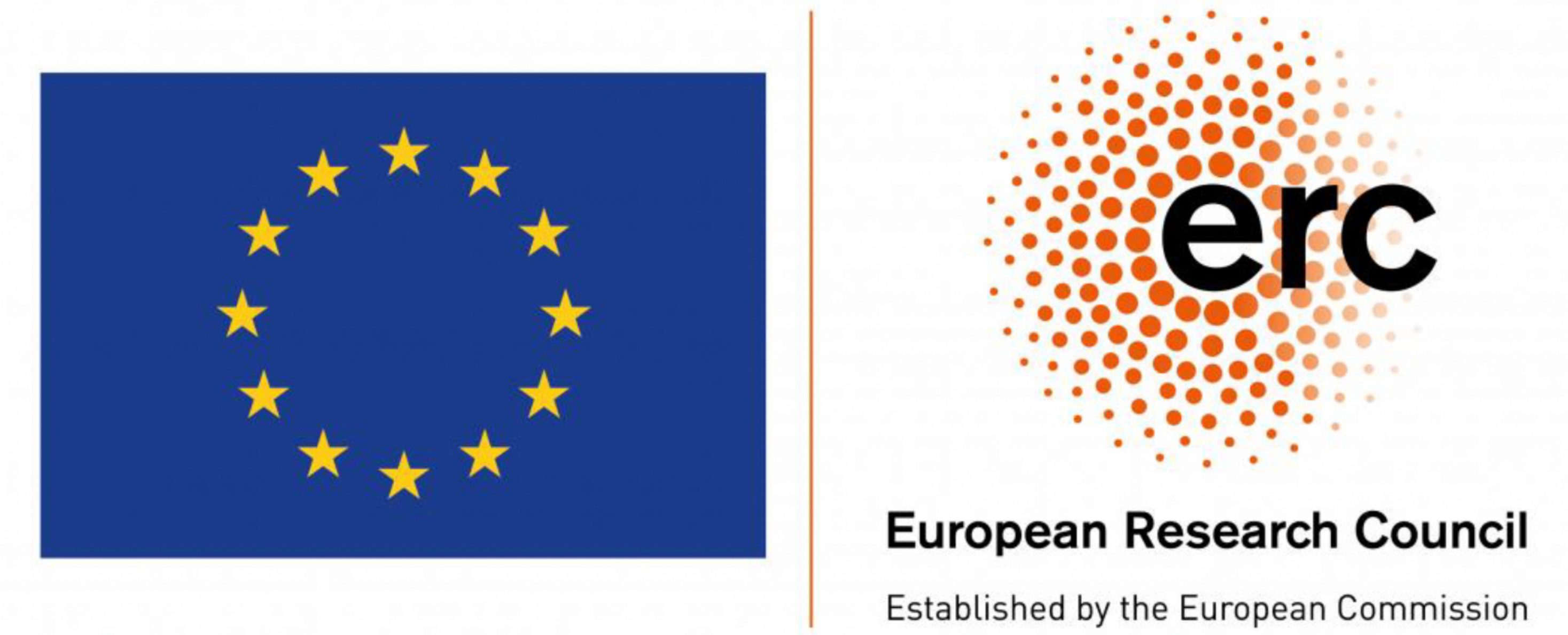}
    }
}
\predate{}
\postdate{}
\date{}

\begin{document}

\maketitle

\begin{abstract}
    We show how derived completion can be used to prove an analogue of Atiyah-Segal completion for the $T$-equivariant Hermitian K-theory of a scheme $X$ with a trivial $T$-action, containing $\tfrac{1}{2}$ and satisfying the resolution property, where $T$ is a split torus of rank $t$. 
    This result is an important first step towards a more general Atiyah-Segal completion theorem for Hermitian K-theory.
\end{abstract}

\section{Introduction}
    \label{section:introduction}

It is a consequence of completion theorems in the style of Atiyah-Segal \cite{atiyah69} that completions of certain algebraic objects correspond to certain geometric constructions, known as \emph{Borel constructions}, which generalize classifying spaces.
The question whether Atiyah-Segal completion holds for Hermitian $\K$-theory is an important one in the development of this invariant. 
In \cite{krishna18equi}, a completion theorem for algebraic K-theory is proved.
In \cite{tabuada21}, a motivic Atiyah-Segal completion theorem is proved for all localizing invariants, of which algebraic $\K$-theory is the universal one by \cite{blumberg13}. 
It is shown in \cite{calmes20hermitian2} that Hermitian $\K$-theory is the universal \emph{quadratic} localizing invariant. 
This quadratic nature introduces some obstacles when trying to prove Atiyah-Segal completion, the first of which are overcome in this article.
Notably, the \emph{Hermitian augmentation ideal} of definition \ref{def:hermitianaugmentationideal} is more complicated than the $\K$-theoretic augmentation ideal used to complete algebraic $\K$-theory.
The results here generalize some of those of the author's thesis \cite[section 8]{rohrbach21}.
In forthcoming work, we hope to prove a more general Atiyah-Segal completion theorem for Hermitian $\K$-theory.

Completion theorems are a way of testing the consistency of equivariant cohomology theories.
The Borel construction involved in the completion theorem provides a feasible way of computing equivariant cohomology, at least up to completion. 
Furthermore, the study of the completion theorem yields valuable insights into the behavior of the equivariant theory and the nature of the fundamental ideal.

We work over a base field $k$ of characteristic not $2$, and we let $X$ be a scheme over $k$ with the resolution property. 
We let $T$ be a split torus of rank $t$ over $k$ acting trivially on $X$.
Our main result is the following, see theorem \ref{theorem:ascompletionfortrivialactionofsplittorus}; the notation is explained in the course of this paper.

\begin{maintheorem} \label{maintheorem:ascompletionforsplittorus}
The natural map
\begin{equation*}
    \pi_{i} \GW^{T,[n]}(X)^{\wedge}_{IO} \lra \pi_i \GW^{[n]}_{\Bo}(X)
\end{equation*}
is an isomorphism for all $i, n \in \Z$.
\end{maintheorem}

The proof hinges on similar results for algebraic $\K$-theory, which are proved in \cite{krishna18equi} and \cite{tabuada21}.
Then, using a common technique for Hermitian $\K$-theory called \emph{Karoubi induction}, these results are extended to Hermitian $\K$-theory.
To do so, a careful examination of the \emph{Hermitian augmentation ideal} is required, as well as the implementation of \emph{derived completion} as set forth in \cite{lurie18}. 

Two obvious next steps would be to prove the Atiyah-Segal completion theorem for trivial actions of more general algebraic groups, starting with connected split reductive groups, and to prove it for non-trivial actions of a split torus, using the idea of \emph{filtrable schemes}.

We expect that Hermitian $\K$-theory does not satisfy Atiyah-Segal completion for either nonreductive groups or singular schemes with an action of a reductive group \cite[section 1.3]{krishna18equi}, but that it does hold for connected split reductive groups acting on a smooth projective $k$-scheme, as long as the scheme is filtrable by a maximal torus of the group.
We use the framework of dg categories as described in \cite{schlichting17} for the definition of the Grothendieck-Witt spectrum and general results such as additivity and localization.

\subsection*{Acknowledgements}
    \label{subsection:acknowledgements}
    
I am grateful to Jens Hornbostel, Marc Levine and Vasudevan Srinivas for helpful discussions and suggestions.

\section{Equivariant Hermitian K-theory}
    \label{section:equivarianthermitianktheory}
    
In this section, we study the $T$-equivariant Hermitian $\K$-theory of schemes, where $T$ is a split torus, after giving a definition for equivariant Hermitian $\K$-theory for a more general class of group schemes.
The main result is lemma \ref{lemma:hermitianidealgenerators}, which gives an explicit finite set of generators of the Hermitian augmentation ideal.
In the next section, these generators are used to study the derived completion in the sense of \cite[section 7.3]{lurie18} of the $T$-equivariant Grothendieck-Witt spectrum of a scheme $X$ with a trivial $T$-action.

\subsection{Equivariance for split tori}
    \label{subsection:equivarianceforsplittori}
    
Let $G$ be a finitely presented, separated, faithfully flat group scheme over a separated noetherian base scheme $S$.
We let $X$ be a scheme with a $G$-action $\theta: G \times X \ra X$ such that $\tfrac{1}{2} \in \OO_X(X)$ and $X$ satisfies the resolution property.
Let $\sPerf^G(X)$ be the pretriangulated dg category of bounded complexes of $G$-equivariant finite locally free $\OO_X$-modules.
Given a line bundle $\mc{L}$ on $X$, $\sPerf^G(X)^{[\mc{L}]}$ denotes the dg category with duality induced by $\mc{L}$. 
%Let $\QCoh^{b,G}_c(X)$ be the full dg subcategory of complexes with coherent cohomology.
If $G$ is trivial, we omit it from the notation.
The central object of study in this section is given by the following definition.

\begin{definition} \label{definition:equivarianthermitianktheory}
Let $\mc{L}$ be a line bundle on $X$.
For $n \in \Z$, let
\begin{equation*}
    \GW^{G,[n]}(X, \mc{L}) = \GW^{[n]}(\sPerf^G(X)^{[\mc{L}]})
\end{equation*}
be the \emph{$G$-equivariant Grothendieck-Witt spectrum of $X$}.
\end{definition}

It is possible to extend the above definition to sheaves of groups on the category of schemes with an appropriate topology, but since we are interested in algebraic groups the definition is general enough.

In fact, we concern ourselves only with the case where $G$ is a split torus.
Representations of split tori over a field $k$ correspond to multi-graded vector spaces.
This phenomenon is not limited to representations, but extends to equivariant sheaves when the base scheme is not the spectrum of some field $k$ anymore. 
This significantly simplifies computations of equivariant cohomology, which is why results are often first proved for split tori, cf. \cite{atiyah69} and \cite{krishna18equi}. 

Fix a base scheme $S$ and let $T$ be a split torus of rank $t$ over $S$.
Let $X$ be a scheme over $S$ with the resolution property and equip it with the trivial $T$-action.
Since $T$-equivariant sheaves of $\OO_X$-modules correspond to $\Z^t$-graded sheaves of $\OO_X$-modules (cf. \cite[proposition 1.1.17]{conrad14}),
a finite locally free $\OO_X$-module $\F$ equipped with a $T$-equivariant structure decomposes as a direct sum
\begin{equation*}
    \F = \bigoplus_{\lambda \in \Z^t} \F_{\lambda}
\end{equation*}
of finite locally free $\OO_X$-modules $\F_{\lambda}$, where all but finitely many of the $\F_{\lambda}$ are zero.
For such $\F$, denote by $W_{\F} \subset \Z^t$ the subset of all $\lambda \in \Z^t$ such that $\F_{\lambda} \neq 0$.
These $\lambda$ are called the \emph{weights of $\F$}.
Note that the trivial action induces the trivial $\Z^t$-grading on $\OO_X$, so that it is concentrated in multi-degree $(0,\dots,0) \in \Z^t$.

We write $\Vect^T(X)$ for the exact category of $T$-equivariant finite locally free $\OO_X$-modules, and $\mc{A} = \sPerf^T(X)^{[\OO_X]}$ for the corresponding dg category with duality of perfect complexes of $T$-equivariant finite locally free $\OO_X$-modules.
If $S = \Spec k$ for some field $k$, then a $T$-equivariant finite locally free sheaf on $S$ is a representation of $T$, which is equivalent to $\Z^t$-graded vector space.

The following proposition gives a semi-orthogonal decomposition of $\mc{A}$ to facilitate computations of its K-theory and GW-theory.

\begin{proposition} \label{proposition:orthogonaldecompositionsplittoriequivariantperfectcomplexes}
For $\lambda \in \Z^t$, let $\mc{A}_{\lambda}$ be the pretriangulated dg subcategory of $\mc{A}$ consisting of perfect complexes of $T$-equivariant locally free $\OO_X$-modules concentrated in multi-degree $\lambda$.
The following statements hold:
\begin{enumerate}[label=(\roman*)]
    \item for each $\lambda \in \Z^t$, $\mc{A}_{\lambda}$ is equivalent to $\sPerf(X)$ as a dg category (without duality);
    \item for $\mu, \lambda \in \Z^t$ such that $\mu \neq \lambda$, $M \in \mc{A}_{\mu}$ and $L \in \mc{A}_{\lambda}$,
        \begin{equation*}
            [M,L]^T = 0 \qquad \text{and} \qquad [L,M]^T = 0,
        \end{equation*}    
        where $[-,-]^T$ is the internal mapping complex of $\mc{A}$; and
    \item there is a semi-orthogonal decomposition $\mc{A} = \langle \mc{A}_{\lambda} \mid \lambda \in \Z^t \rangle$.
\end{enumerate}
\end{proposition}

\begin{proof}
For $\lambda \in \Z^t$, the functor $\sPerf(X) \ra \mc{A}_{\lambda}$ sending a locally free $\OO_X$-module $\F$ to the $T$-equivariant locally free $\OO_X$-module $\F(\lambda)$ such that $W_{\F(\lambda)} = \{\lambda\}$ and $\F(\lambda)_{\lambda} = \F$ is an equivalence.

Furthermore, $T$-equivariant maps $\F \ra \G$ of $T$-equivariant locally free $\OO_X$-modules respect the induced $\Z^t$-gradings on $\F$ and $\G$.
So if $W_{\F} \cap W_{\G} = \emptyset$, then $\sheafhom^T(\F, \G) = 0$, and this extends to perfect complexes.

Lastly, every object $M$ of $\mc{A}$ decomposes as a direct sum
\begin{equation*}
    M = \bigoplus_{\lambda \in \Z^t} M_{\lambda}
\end{equation*}
with $M_{\lambda}$ in $\mc{A}_{\lambda}$ for all $\lambda \in \Z^t$. 
\end{proof}

This yields the following computation of the representation ring of $T$. 

\begin{corollary} \label{corollary:ktheorysplittorusequivariant}
There is an isomorphism of rings
\begin{equation*}
    \K^T_0(X) \cong \frac{\K_0(X)[x_1, \dots, x_t, y_1, \dots, y_t]}{(x_iy_i + x_i + y_i \mid i=1,\dots, t)}.
\end{equation*}
\end{corollary}

\begin{proof}
Note that the product on $\K^T_0(X)$ is induced by the tensor product of $\OO_X$-modules.
By additivity for $\K$-theory and proposition \ref{proposition:orthogonaldecompositionsplittoriequivariantperfectcomplexes}, there is an isomorphism of $\K_0(X)$-modules
\begin{equation*}
    \bigoplus_{\lambda \in \Z^t} \K_0(X) \lra \K^T_0(X).
\end{equation*}
For $\lambda \in \Z^t$, let $\OO_X(\lambda)$ be the $T$-equivariant $\OO_X$-module with $W_{\OO_X(\lambda)} = \{ \lambda \}$ and $\OO_X(\lambda)_{\lambda} = \OO_X$.
For $1 \leq i \leq n$, let $e_i \in \Z^t$ be the $i$-th unit vector, and write $x_i$ and $y_i$ for the $\K$-theory classes $[\OO_X(e_i)] - 1$ and $[\OO_X(-e_i)] - 1$, respectively.  
Then $x_iy_i + x_i + y_i = 0$ for all $i$. 
Setting $\lambda = (\lambda_1, \dots, \lambda_t)$, we see that
\begin{equation*}
    \OO_X(\lambda) = \prod_{i=1}^t (x_i + 1)^{\lambda_i},
\end{equation*}
where $(x_i + 1)^{\lambda_i} = (y_i + 1)^{-\lambda_i}$ if $\lambda_i < 0$. 
Hence
\begin{equation*}
    \K^T_0(X) \cong \frac{\K_0(X)[x_1, \dots, x_t, y_1, \dots, y_t]}{(x_iy_i + x_i + y_i \mid i=1,\dots, t)},
\end{equation*}
as was to be shown.
\end{proof}

Let $(\Z^t - \{0\})/\{\pm\}$ be the quotient of $\Z^t - \{0\}$ by the sign involution.
A useful system of representatives $C$ of this is xgiven by nonzero $(a_1, \dots, a_t) \in \Z^t$ such that the first nonzero entry $a_i$ is positive; if $t = 1$, $C$ consists of the positive integers.

\begin{corollary} \label{corollary:gwtheorysplittorusequivariant}
Let $\K(X)_{{\lambda}} = \K(\mc{A}_{\lambda})$ for all ${\lambda} \in C$. 
For each $i,n \in \Z$, the map of $\GW^{[0]}_0(X)$-modules
\begin{equation} \label{equation:grothendieckwittspectrumtrivialsplittorusaction}
    \GW^{[n]}_i(X) \oplus \bigoplus_{{\lambda} \in C} \K_i(X)_{\lambda} 
        \lra \GW^{T,[n]}_i(X)
\end{equation}
induced by the dg form functor
\begin{equation*}
    \begin{aligned}
        \mc{A}_0 \times \bigoplus\limits_{\lambda \in C} H\mc{A}_{\lambda}
            & \lra \mc{A} \\
        A_0 \times \bigoplus\limits_{\lambda \in C} (A_{\lambda}, B_{\lambda})  
            & \lra A_0 \oplus \left(\bigoplus\limits_{\lambda \in C} A_{\lambda} \oplus B_{\lambda}^{\vee}\right)
    \end{aligned}
\end{equation*}
is an isomorphism.
In particular, the map
\begin{equation} \label{equation:splittorusgrothendieckwittring}
    \begin{aligned}
        \GW^{[0]}_0(X) \oplus \bigoplus_{{\lambda} \in C} \K_0(X)_{\lambda} 
            & \lra \GW^{T, [0]}_0(X) \\
        (a, (b_{{\lambda}})_{{\lambda}}) 
            & \longmapsto a(0) + \sum_{{\lambda} \in C} \left(H_0(b_{{\lambda}}(\lambda)) - 2 \right)
    \end{aligned}
\end{equation}
is an isomorphism.
\end{corollary}

\begin{proof}
Let $\lambda \in \Z^t$.
If $\lambda = 0$, then $\mc{A}_{\lambda}$ is fixed by the standard duality on $\mc{A}$.
Otherwise, $(\mc{A}_{\lambda})^{\vee} = \mc{A}_{-\lambda}$. 
Thus, letting $\mc{A}_+ = \langle \mc{A}_{\lambda} \mid \lambda \in C \rangle$, there is a semi-orthogonal decomposition $\mc{A} = \langle \mc{A}_+^{\vee}, \mc{A}_0, \mc{A}_+ \rangle$, and the result follows from proposition \ref{proposition:orthogonaldecompositionsplittoriequivariantperfectcomplexes} and additivity for Grothendieck-Witt theory \cite[proposition 6.8]{schlichting17}.

Furthermore, there is an automorphism on
\begin{equation*}
    \GW^{[0]}_0(X) \oplus \bigoplus_{{\lambda} \in C} \K_0(X)_{\lambda},
\end{equation*}
which is the identity on $\GW^{[0]}_0(X)$ and given by $[\OO_X(\lambda)] \mapsto [\OO_X(\lambda)] - 1$ for $\lambda \in C$. 
By composing the map of $\GW^{[0]}_0(X)$-modules induced by (\ref{equation:grothendieckwittspectrumtrivialsplittorusaction}) with this automorphism, one obtains the $-2$ term on the right hand side of (\ref{equation:splittorusgrothendieckwittring}). 
\end{proof}

The following two definitions are instrumental in the statement and proof of Atiyah-Segal completion for split tori.

\begin{definition} \label{def:hermitianaugmentationideal}
Let $RO = \GW^{T,[0]}_0(X)$. 
The map $\alpha: RO \ra \GW^{[0]}_0(X)$ which forgets the equivariant structure is called the \emph{Hermitian augmentation map}.
Its kernel $IO = \ker \alpha$ is called the \emph{Hermitian augmentation ideal}.
\end{definition}

\begin{definition} \label{def:reducedgrothendieckwittgroups}
For any $n,i \in \Z$, the kernel of the map $\GW^{T,[n]}_i(X) \ra \GW^{[n]}_i(X)$ is the \emph{reduced equivariant Grothendieck-Witt group} $\widetilde{\GW}{}^{T,[n]}_i(X)$.
The groups $\widetilde{K}{}^T_i(X)$ and $\widetilde{\W}{}^{T,[n]}(X)$ are defined similarly.
\end{definition}

For $R = \K^T_0(X)$, the augmentation map $\alpha: R \ra \K_0(X)$ has kernel $I$, which is given by $(x_1, \dots, x_t, y_1, \dots, y_t)$ under the isomorphism of corollary \ref{corollary:ktheorysplittorusequivariant}.
For $\lambda, \mu \in \N^t$, let 
\begin{equation*}
    \bm{x}^{\lambda}\bm{y}^{\mu} = \prod_{i=1}^t x_i^{\lambda_i}y_i^{\mu_i}.
\end{equation*}
The identity $x_iy_i = -(x_i + y_i)$ on $R$ shows that $R$ is generated as a $\K_0(X)$-module by monomials $\bm{x}^{\lambda}\bm{y}^{\mu}$ with $\lambda, \mu \in \N^t$ such that, for each $i$, either $\lambda_i = 0$ or $\mu_i = 0$. 

By corollary \ref{corollary:gwtheorysplittorusequivariant},
\begin{equation*}
    IO \cong \bigoplus_{{\lambda} \in C} \K_0(X)
\end{equation*}
under the isomorphism (\ref{equation:splittorusgrothendieckwittring}). 
The standard duality on $\mc{A}$ induces an involution $\vee :R \ra R$, given by $x_i \mapsto y_i$. 
Let $F_0: \GW^{T,[0]}_0(X) \ra \K^T_0(X)$ be the forgetful map. 
Note that $F_0$ restricts to a map $F_0: IO \ra I$. 
The following lemma shows that $RO$ splits as $\GW^{[0]}_0(X)$ and the $\vee$-fixed points of $I$.

\begin{lemma} \label{lemma:hermitianidealisfixedpointsofaugmentationideal}
Let $I_+$ be the $\vee$-fixed points of $I$.
The map $F_0: IO \ra I$ is injective with image $I_+$.
\end{lemma}

\begin{proof}
Consider the following diagram
\begin{equation*}
    \begin{tikzcd}
        \widetilde{\GW}{}^{T,[3]}_{0}(X) \arrow[r, "F_3"] \arrow[d]
            & \widetilde{\K}{}^T_{0}(X) \arrow[r, "H_0"] \arrow[d]
                & \widetilde{\GW}{}^{T,[0]}_{0}(X) \arrow[r] \arrow[d]
                    & \widetilde{\W}{}^{T,[0]}_{0}(X) \arrow[d] \\
        {\GW}^{T,[3]}_{0}(X) \arrow[r, "F_3"] \arrow[d]
            & {\K}^T_{0}(X) \arrow[r, "H_0"] \arrow[d]
                & {\GW}^{T,[0]}_{0}(X) \arrow[r] \arrow[d]
                     & {\W}^{T,[0]}(X) \arrow[d] \\
        \GW^{[3]}_{0}(X) \arrow[r, "F_3"]
            & \K_{0}(X) \arrow[r, "H_0"]
                & \GW^{[0]}_{0}(X) \arrow[r]
                     & \W^{[0]}(X)  
    \end{tikzcd}
\end{equation*}
of Karoubi sequences \cite[theorem 6.1]{schlichting17} of $RO$-modules.
Note that $\widetilde{\K}{}^T_0(X) = I$ and $\widetilde{\GW}{}^{T,[0]}_0(X) = IO$.
The image of the forgetful map $F_0: IO \ra I$ is necessarily contained in $I_+$. 
Furthermore, $I_+$ is generated as a $\K_0(X)$-module by elements of the form $\bm{x}^{\lambda}\bm{y}^{\mu} + \bm{x}^{\mu}\bm{y}^{\lambda}$, where $\lambda, \mu \in \N^t$. 
Define a map $G_0: I_+ \ra IO$ by $\bm{x}^{\lambda}\bm{y}^{\mu} + \bm{x}^{\mu}\bm{y}^{\lambda} \mapsto H_0(\bm{x}^{\lambda}\bm{y}^{\mu})$.
This map is well-defined since $H_0(\bm{x}^{\lambda}\bm{y}^{\mu}) = H_0(\bm{x}^{\mu}\bm{y}^{\lambda})$.

Since $\W^{T,[n]}(X) \cong \W^{[n]}(X)$ for $n \in \Z$, it follows that $\widetilde{\W}{}^{T,[0]}(X) = 0$.
Therefore, the hyperbolic map $H_0: I \ra IO$ is surjective, and one may write $F_0: H_0(I) \ra I_+$ and $G_0: I_+ \ra H_0(I)$.
Upon inspection, $F_0$ and $G_0$ are inverse to each other, which concludes the proof.
\end{proof}

We want to complete Grothendieck-Witt spectra and K-theory spectra with respect to $IO$.
The following lemma states that $IO$ is finitely generated, which ensures that these completions are sufficiently well-behaved.

\begin{lemma} \label{lemma:hermitianidealgenerators}
As an ideal of $R_+$, $I_+$ is generated by elements of the form $\bm{x}^{\gamma} + \bm{y}^{\gamma}$, where $\gamma \in \{0,1\}^t - \{0\}^t$.
In particular, $I_+$ is a finitely generated ideal with a generating set of $2^t-1$ elements.
\end{lemma}

\begin{proof}
Let $J$ be the ideal generated by elements of the form $\bm{x}^{\gamma} + \bm{y}^{\gamma}$, where $\gamma \in \{0,1\}^t - \{0\}^t$.
Since $I_+$ is generated as a free abelian group by elements of the form $\bm{x}^{\lambda}\bm{y}^{\mu} + \bm{x}^{\mu}\bm{y}^{\lambda}$ with $\lambda, \mu \in \N^t$, it suffices to show that $J$ contains these elements.

First, it will be shown by induction on $n$ that $J$ contains all elements of the form $\bm{x}^{\lambda} + \bm{y}^{\lambda}$, where $\lambda \in \{0,\dots,n\}^t$ for arbitrary $n \in \N$.
For $n=1$, this follows directly from the definition of $J$.
Now assume that $\bm{x}^{\lambda} + \bm{y}^{\lambda} \in J$ for all $\lambda \in \{0,\dots,n\}^t$ and let $\lambda \in \{0, \dots, n+1\}^t$. 
Then $\lambda = \mu + \gamma$ with $\mu \in \{0,\dots,n\}^t$ and $\gamma \in \{0,1\}^t$ such that $\gamma_i = 1$ if and only if $\lambda_i = n+1$.
Note that
\begin{equation*}
    (\bm{x}^{\mu} + \bm{y}^{\mu})(\bm{x}^{\gamma} + \bm{y}^{\gamma}) = \bm{x}^{\lambda} + \bm{y}^{\lambda} + \bm{x}^{\mu}\bm{y}^{\gamma} + \bm{x}^{\gamma}\bm{y}^{\mu},
\end{equation*}
and $(\bm{x}^{\mu} + \bm{y}^{\mu})(\bm{x}^{\gamma} + \bm{y}^{\gamma}) \in J$.
Moreover, $\mu - \gamma \in \{0,\dots,n\}^t$, so
\begin{equation*}
    \bm{x}^{\mu}\bm{y}^{\gamma} + \bm{x}^{\gamma}\bm{y}^{\lambda} = 
    \bm{x}^{\mu-\gamma}\bm{x}^{\gamma}\bm{y}^{\gamma} + \bm{x}^{\gamma}\bm{y}^{\gamma}\bm{y}^{\mu-\gamma} =
    (\bm{x}^{\mu-\gamma} + \bm{y}^{\mu-\gamma}) \prod_{i=1}^t(-(x_i + y_i))^{\gamma_i}
\end{equation*}
is contained in $J$.
By induction, it follows that $\bm{x}^{\lambda} + \bm{y}^{\lambda} \in J$ for all $\lambda \in \N^t$, as claimed.

Now let $\lambda, \mu \in \N^t$ and consider $\bm{x}^{\lambda}\bm{y}^{\mu} + \bm{x}^{\mu}\bm{y}^{\lambda}$.
Note that
\begin{equation*}
    (\bm{x}^{\lambda} + \bm{y}^{\lambda})(\bm{x}^{\mu} + \bm{y}^{\mu}) = 
    \bm{x}^{\lambda + \mu} + \bm{y}^{\lambda + \mu} + \bm{x}^{\lambda}\bm{y}^{\mu} + \bm{x}^{\mu}\bm{y}^{\lambda}
\end{equation*}
is contained in $J$, and since $\bm{x}^{\lambda + \mu} + \bm{y}^{\lambda + \mu} \in J$, this yields $\bm{x}^{\lambda}\bm{y}^{\mu} + \bm{x}^{\mu}\bm{y}^{\lambda} \in J$.
\end{proof}

\section{The completion theorem}
    \label{section:thecompletiontheorem}

We study completions with respect to the Hermitian augmentation ideal in section \ref{subsection:completinghermitianaugmentationideal} and prove the completion theorem for the $T$-equivariant Hermitian $\K$-theory of a scheme $X$ with a trivial action of a split torus $T$ in section \ref{subsection:thecompletiontheoremfortrivialactions}.
The idea is to prove the result for triangular Witt groups (lemma \ref{lemma:ascompletionforwittgroups}) and then employ Karoubi induction to prove our main theorem \ref{theorem:ascompletionfortrivialactionofsplittorus}. 

\subsection{Completing with respect to the Hermitian augmentation ideal}
    \label{subsection:completinghermitianaugmentationideal}

We first show that the two completions of $T$-equivariant $\K$-theory with respect to the augmentation ideal and the Hermitian augmentation ideal agree.
The following result is \cite[lemma 8.2.12]{rohrbach21}.

\begin{lemma}[Filtration lemma] \label{lemma:filtrationofaugmentationideal}
For a linear algebraic group $G$ over a field $k$ of characteristic not $2$, let $R_G = \K^G_0(k)$ and $RO_G = \GW^{G,[0]}(k)$ with $I_G \subset R_G$ and $IO_G \subset RO_G$ the $\K$-theoretic augmentation ideal and Hermitian augmentation ideal, respectively.
For any linear algebraic group $G$ over $k$, the $I_G$-adic and $IO_G$-adic topologies on $R_G$ coincide.
\end{lemma}

\begin{proof}
Fix an embedding $\iota: G \ra H$ with $H = \SO_{2m+1}$ for some $m \in \N$, which can be realized as the composition of embeddings
\begin{equation*}
    G \lra \GL_m \lra \SO_{2m} \lra \SO_{2m+1}.
\end{equation*}
Let $F_G: RO_G \ra R_G$ be the forgetful map.
It will be shown that
\begin{equation} \label{equation:filtrationlemmaaugmentationidealinclusions}
    \iota^*(I_H)R_G \subset F_G(IO_G) R_G \subset I_G R_G,
\end{equation}
after which it suffices to show that the $\iota^*(I_H)$-adic and $I_G$-adic topologies on $R_G$ coincide.

The second inclusion of (\ref{equation:filtrationlemmaaugmentationidealinclusions}) follows from the commutativity of 
\begin{equation*}
    \begin{tikzcd}
        RO_G \arrow[r, "F_G"] \arrow[d]
            & R_G \arrow[d] \\
        \GW_0(k) \arrow[r, "F"]
            & \K_0(k).
    \end{tikzcd}
\end{equation*}
Since $H$ is split reductive and all irreducible representations of $H$ are symmetric by \cite[lemma 3.14]{zibrowius15}, the forgetful map $F_H: RO_H \ra R_H$ is surjective.
The trivial map $\phi: H \ra H$ given by $g \mapsto 1$ induces morphisms $\phi^*: R_H \ra R_H$ and $\phi^*: RO_H \ra RO_H$ which replace any $H$-representation by the trivial one.  
As elements of $I_H$ are of the form $a - b$ with $\phi^*a = \phi^*b$, the map $1 - \phi^*: R_H \ra I_H$ splits the inclusion $I_H \ra R_H$. 
Similarly, $1 - \phi^*: RO_H \ra IO_H$ splits $IO_H \ra RO_H$. 
Thus the commutative diagram
\begin{equation*}
    \begin{tikzcd}
        RO_H \arrow[r, two heads, "F_H"] \arrow[d, two heads, swap, "1 - \phi^*"]
            & R_H \arrow[d, two heads, "1 - \phi^*"] \\
        IO_H \arrow[r, "F_H"] 
            & I_H
    \end{tikzcd}
\end{equation*}
shows that $F_H: IO_H \ra I_{H}$ is surjective.
Consequently, $\iota^*(I_H) = \iota^*(F_H(IO_H))$.
Therefore, the commutativity of
\begin{equation*}
    \begin{tikzcd}
        IO_H \arrow[dd] \arrow[dr] \arrow[rr]
            &[-1em] { }
                & I_H \arrow[d] \\[-1em]
        { }
            & RO_H \arrow[r, "F_H"] \arrow[d, "\iota^*"]
                & R_H \arrow[d, "\iota^*"] \\
        IO_G \arrow[r]
            & RO_G \arrow[r, "F_G"]    
                & R_G
    \end{tikzcd}
\end{equation*}
shows that $\iota^*(I_H) \subset F_G(IO_G)$, which proves (\ref{equation:filtrationlemmaaugmentationidealinclusions}).
Hence, the $\iota^*(I_H)$-adic and $I_G$-adic topologies on $R_G$ coincide by \cite[corollary 6.1]{edidin00}, and the result follows. 
\end{proof}

We obtain the following derived variant of the filtration lemma as a corollary.

\begin{lemma}[Derived filtration lemma] \label{lemma:derivedfiltrationlemma}
For an algebraic group $G$ over a field $k$ of characteristic not $2$ acting on a scheme $X$ over $k$ with the resolution property, the derived completions $\K^G(X)^{\wedge}_{IO}$ and $\K^G(X)^{\wedge}_{I}$ are equivalent.
\end{lemma}

\begin{proof}
Let $R = \K^G_0(k)$ and $RO = \GW^{G,[0]}_0(k)$.
Let $F: RO \ra R$ be the forgetful map and let $I'$ be the ideal of $\K^T_0(k)$ generated by $F(IO)$.
By the commutativity of
\begin{equation*}
    \begin{tikzcd}
        RO \arrow[r, "F"] \arrow[d]
            & R \arrow[d] \\
        \GW^{[0]}_0(k) \arrow[r, "F"]
            & \K_0(k),
    \end{tikzcd}
\end{equation*}
we have an inclusion $I' \subset I$, as $IO$ is the kernel of the left vertical map.
Hence there is a natural map $\K^T(X)^{\wedge}_{I'} \ra \K^T(X)^{\wedge}_{I}$.
We will show that this map is an equivalence of spectra by showing that it is an isomorphism on the homotopy groups.
By \cite[proposition 7.3.6.6]{lurie18} and the fact that $R$ is a Noetherian ring by corollary \ref{corollary:ktheorysplittorusequivariant}, it follows that $\pi_i \K^T(X)^{\wedge}_{I'}$ and $\pi_i \K^T(X)^{\wedge}_{I}$ coincide with the $I'$-adic and $I$-adic completions of $\K^T_i(X)$.
By the filtration lemma \ref{lemma:filtrationofaugmentationideal}, the natural map
\begin{equation*}
    \Cpl(\K^T_i(X), I') \lra \Cpl(\K^T_i(X), I)
\end{equation*}
is an isomorphism, as was to be shown.
\end{proof}

The following proposition, suggested to the author by Marc Levine, is useful for determining the completeness of certain modules we are about to consider.

\begin{proposition} \label{proposition:nilpotentclasses}
Let $X$ be a quasi-compact scheme with a line bundle $\mc{L}$.
Let $H_n^{\mc{L}}(E) \in \GW^{[n]}(X, \mc{L})$, where $E$ is a vector bundle on $X$ of rank $r$. 
Then $H_n^{\mc{L}}(E) - rH_n^{\mc{L}}(\OO)$ is nilpotent in the total Grothendieck-Witt ring of $X$. 
\end{proposition}

\begin{proof}
Let $\{U_i\}_{i \in I}$ be a cover of $X$ that trivializes $E$. 
Since $X$ is quasi-compact, we may assume $I$ finite.
For each $i \in I$, let $Z_i = X - U_i$ and let $a_i: U_i \ra X$ be the inclusion.
Then there are localization sequences
\begin{equation*}
    \GW^{[n]}(X\ \tn{on}\ Z_i, \mc{L}) \lra \GW^{[n]}(X,\mc{L}) \lra \GW^{[n]}(U_i,\mc{L})
\end{equation*}
of commutative ring spectra.
Note $H_n^{\mc{L}}(E) - rH_n^{\mc{L}}(\OO)$ is zero in $\GW^{[n]}_0(U_i,\mc{L})$ for all $i \in I$.
Thus for $i \in I$, let $[E_i, q_i] \in \GW^{[n]}(X, \mc{L})$ be a preimage of $H_n^{\mc{L}}(E) - rH_n^{\mc{L}}(\OO)$ with support in $Z_i$.
The product
\begin{equation*}
    \prod_{i \in I} [E_i, q_i]
\end{equation*}
has support in the intersection of all the $Z_i$, which is empty, whence the product vanishes.
\end{proof}

Now let $k$ be a field of characteristic not two.    
Let $T$ be a split torus of rank $t$ over $k$.
Let $RO = \GW^{T,[0]}_0(k)$ and $R = \K^T_0(k)$ with the Hermitian augmentation ideal $IO \subset RO$ as in definition \ref{def:hermitianaugmentationideal} and the augmentation ideal $I \subset R$, respectively. 
By lemma \ref{lemma:hermitianidealgenerators}, $IO$ is a finitely generated ideal.
Let $x_1, \dots, x_q \in RO$ be a set of generators of $IO$.
Let $X$ be a regular scheme with a trivial $T$-action.
Let $\mc{A} = \Perf^T(X)^{[\mc{L}]}$ be the pretriangulated dg category of perfect complexes of $T$-equivariant $\OO_X$-modules with duality induced by a $T$-equivariant line bundle $\mc{L}$ on $X$. 
For $r \in \N$, let $X_r = (\P_X^{2r})^t$ and let $X_{\Bo} = \colim_r X_r$, where `Bo' stands for `Borel construction'. 
Let $\mc{L}_r$ be the pullback of $\mc{L}$ along the projection $X_r \ra X$.
Objects of $\mc{A}$ correspond to complexes of $\Z^t$-graded vector bundles on $X$. 

Let $\K_{\Bo}(X) = \lim_{r} \K(X_r)$ and $\GW^{[n]}_{\Bo}(X, \mc{L}) = \lim_r \GW^{[n]}(X_r, \mc{L}_r)$. 
For each $r \in \N$, there is a natural map $\K^T(X)^{\wedge}_{IO} \ra \K(X_r)$ since there is a canonical functor $\Perf^T(X) \ra \Perf(X_r)$ and $\K(X_r)$ is $IO$-complete by the derived filtration lemma \ref{lemma:derivedfiltrationlemma}.
Hence there is a natural map $\theta: \K^T(X)^{\wedge}_{IO} \ra \K_{\Bo}(X)$, which shows that $\K_{\Bo}(X)$ is $IO$-complete, cf. \cite[theorem 1.4(i)]{tabuada21}.
Furthermore, $\theta$ is an equivalence by lemma \ref{lemma:derivedfiltrationlemma} and \cite[theorem 1.4(ii)]{tabuada21}.

\begin{lemma} \label{lemma:gwofprojectivespaceiscomplete}
For $r \in \N$, $n \in \Z$ and $\mc{L}'$ a line bundle on $X_r$, $\GW^{[n]}(X_r, \mc{L}')$ is $IO$-complete in the sense of \cite[definition 7.3.1.1]{lurie18}.
\end{lemma}

\begin{proof}
By \cite[proposition 7.3.2.1]{lurie18}, it suffices to show that the limit of the tower
\begin{equation*}
    \dots \stlra{\cdot x_j} \GW^{[n]}(X_r, \mc{L}') \stlra{\cdot x_j} \GW^{[n]}(X_r, \mc{L}')
\end{equation*}
vanishes for each $1 \leq j \leq q$. 
The elements $x_j$ are of the form $H_0(k(\lambda) - k)$ so they act on $\GW^{[n]}(X_r, \mc{L}')$ as $H_0(\OO_{X_r}(\lambda)) - H_0(\OO) \in \GW^{[0]}_0(X_r)$. 
By proposition \ref{proposition:nilpotentclasses}, these classes are nilpotent, so the limit vanishes and the result follows.
\end{proof}

\begin{corollary} \label{corollary:borelconstructioniscomplete}
For $r \in \N$ and $n \in \Z$, $\GW^{[n]}_{\Bo}(X, \mc{L})$ is $IO$-complete in the sense of \cite[definition 7.3.1.1]{lurie18}.
In particular, there is a natural map $\gamma: \GW^{T,[n]}(X,\mc{L})^{\wedge}_{IO} \ra \GW^{[n]}_{\Bo}(X, \mc{L})$. 
\end{corollary}

\begin{proof}
The spectrum $\GW^{[n]}_{\Bo}(X, \mc{L}) = \lim_r \GW^{[n]}(X_r, \mc{L}_r)$ is a limit of $IO$-complete spectra, which is $IO$-complete by \cite[proposition 7.3.1.4]{lurie18}.
Hence there is a canonical map $\gamma: \GW^{T,[n]}(X,\mc{L})^{\wedge}_{IO} \ra \GW^{[n]}_{\Bo}(X, \mc{L})$. 
\end{proof}

\subsection{The completion theorem for trivial actions}
    \label{subsection:thecompletiontheoremfortrivialactions}
    
We keep the notation of the previous section.
For each $r \in \N$, there is a commutative diagram of Bott sequences
\begin{equation*}
    \begin{tikzcd}
        \GW^{T,[n]}(X, \mc{L}) \arrow[r] \arrow[d]
            & \K^T(X) \arrow[r] \arrow[d]
                & \GW^{T,[n+1]}(X,\mc{L}) \arrow[d] \\
        \GW^{[n]}(X_r,\mc{L}_r) \arrow[r]
            & \K(X_r) \arrow[r]
                & \GW^{[n+1]}(X_r,\mc{L}_r)
    \end{tikzcd}
\end{equation*}
and it follows that there is a commutative diagram
\begin{equation*}
    \begin{tikzcd}
        \GW^{T,[n]}(X, \mc{L}) \arrow[r] \arrow[d]
            & \K^T(X) \arrow[r] \arrow[d]
                & \GW^{T,[n+1]}(X, \mc{L}) \arrow[d] \\
        \GW^{[n]}_{\Bo}(X, \mc{L}) \arrow[r]
            & \K_{\Bo}(X) \arrow[r]
                & \GW^{[n+1]}_{\Bo}(X, \mc{L})
    \end{tikzcd}
\end{equation*}
yielding a map of long exact sequences
\begin{equation*}
    \begin{tikzcd}[column sep=small]
        \pi_i \GW^{T,[n]}(X)^{\wedge}_{IO} \arrow[r] \arrow[d]
            & \pi_i \K^T(X)^{\wedge}_{IO} \arrow[r] \arrow[d]
                & \pi_i \GW^{T,[n+1]}(X)^{\wedge}_{IO} \arrow[r] \arrow[d]
                    & \pi_{i-1} \GW^{T,[n]}(X)^{\wedge}_{IO} \arrow[d] \\
        \pi_i \GW^{[n]}_{\Bo}(X, \mc{L}) \arrow[r]
            & \pi_i \K_{\Bo}(X) \arrow[r]
                & \pi_i \GW^{[n+1]}_{\Bo}(X, \mc{L}) \arrow[r]
                    & \pi_{i-1} \GW^{[n]}_{\Bo}(X, \mc{L})
    \end{tikzcd}
\end{equation*}
We will now show by a slightly modified version of Karoubi induction that all the vertical arrows are isomorphisms.

\begin{proposition} \label{proposition:admissiblegadgetsplittorusmittagleffler}
The tower
\begin{equation*}
    \{ \GW^{[n]}_{i+1}(X_r, \mc{L}_r) \}_r
\end{equation*}
satisfies the Mittag-Leffler condition.
\end{proposition}

\begin{proof}
The proof can be directly adapted from \cite[proposition 8.2.2]{rohrbach21}
\end{proof}

\begin{lemma} \label{lemma:ascompletionforwittgroups}
The natural map
\begin{equation*}
    \pi_{i} \GW^{T,[n]}(X, \mc{L})^{\wedge}_{IO} \lra \pi_i\GW^{[n]}_{\Bo}(X, \mc{L})
\end{equation*}
is an isomorphism for all $i \leq -2$ and $n \in \Z$.
\end{lemma}

\begin{proof}
Let $i \leq -2$ and $n \in \Z$.
Note that $\GW^{[n]}_{\Bo}(X, \mc{L}) = \lim_r \GW^{[n]}(X_r, \mc{L}_r)$.
Hence there is a Milnor exact sequence
\begin{equation*}
    0 \lra \limone_r \GW^{[n]}_{i+1}(X_r, \mc{L}_r) \lra \pi_i \GW^{[n]}_{\Bo}(X, \mc{L}) \lra \lim_r \GW^{[n]}_i(X_r, \mc{L}_r) \lra 0
\end{equation*}
As the tower $\{ \GW^{[n]}_{i+1}(X_r, \mc{L}_r) \}_r$ satisfies the Mittag-Leffler condition by proposition \ref{proposition:admissiblegadgetsplittorusmittagleffler}, the $\limone$-term vanishes and the natural map
\begin{equation*}
    \pi_i \GW^{[n]}_{\Bo}(X, \mc{L}) \lra \lim_r \GW^{[n]}_i(X_r, \mc{L}_r)    
\end{equation*}
is an isomorphism. 
By \cite[main theorem A]{rohrbach22}, $\GW^{[n]}_i(X_r, \mc{L}_r) \cong \W^{[n-i]}(X, \mc{L})$, yielding $\pi_i \GW^{[n]}_{\Bo}(X, \mc{L}) \cong \W^{[n-i]}(X, \mc{L})$. 
Note that $\GW^{T,[n]}_i(X, \mc{L}) \cong \W^{[n-i]}(X, \mc{L})$ by corollary \ref{corollary:gwtheorysplittorusequivariant} combined with \cite[proposition 6.3]{schlichting17}. 
It follows that the natural map $\GW^{T,[n]}_i(X, \mc{L}) \ra \pi_i \GW^{[n]}_{\Bo}(X, \mc{L})$ is an isomorphism.
Now it suffices to show that the natural map $\GW^{T,[n]}_i(X) \ra \pi_{i} \GW^{T,[n]}(X)^{\wedge}_{IO}$ is an isomorphism by two-out-of-three for isomorphisms.

Fix $1 \leq j \leq q$. 
By \cite[proposition 7.3.2.1]{lurie18}, 
\begin{equation*}
    \GW^{T,[n]}(X, \mc{L})^{\wedge}_{x_j} \simeq \cofib(T(\GW^{T,[n]}(X), \mc{L}) \ra \GW^{T,[n]}(X), \mc{L}),
\end{equation*}
where $T(\GW^{T,[n]}(X), \mc{L})$ is the limit of the tower
\begin{equation*}
    \dots \stlra{\cdot x_j} \GW^{T,[n]}(X, \mc{L}) \stlra{\cdot x_j} \GW^{T,[n]}(X, \mc{L}).
\end{equation*}
For $r \in \Z$, let $A_r$ be the corresponding tower
\begin{equation*}
    \dots \stlra{\cdot x_j} \pi_r \GW^{T,[n]}(X, \mc{L}) \stlra{\cdot x_j} \pi_r \GW^{T,[n]}(X, \mc{L})
\end{equation*}
of homotopy groups.
Then there is a Milnor exact sequence
\begin{equation*}
    0 \lra \limone A_{r+1} \lra \pi_r T(\GW^{T,[n]}(X), \mc{L}) \lra \lim A_r \lra 0.
\end{equation*}
Since $x_j$ is hyperbolic and $\pi_r \GW^{T,[n]}(X, \mc{L}) \cong \W^{[n-r]}(X, \mc{L})$ whenever $r \leq -1$, multiplication by $x_j$ is the zero map on $\pi_r \GW^{T,[n]}(X, \mc{L})$.
It follows that for $r \leq -2$ the $\limone$-term and the limit term of the Milnor exact sequence vanish, and $\pi_r T(\GW^{T,[n]}(X, \mc{L}))$ vanishes as well.
Using the long exact sequence
\begin{equation*}
    \begin{tikzcd}
        \pi_{r}T(\GW^{T,[n]}(X, \mc{L})) \arrow[r]  
            & \pi_r \GW^{T,[n]}(X, \mc{L}) \arrow[r]
                & \pi_r \GW^{T,[n]}(X, \mc{L})^{\wedge}_{x_j} \arrow[d] \\
        {}
            & {}
                & \pi_{r-1}T(\GW^{T,[n]}(X, \mc{L}))
    \end{tikzcd}
\end{equation*}
we see that $\pi_r \GW^{T,[n]}(X, \mc{L}) \ra \pi_r \GW^{T,[n]}(X, \mc{L})^{\wedge}_{x_j}$ is an isomorphism for all $r \leq -2$. 
By induction on the number of generators $q$ using \cite[proposition 7.3.3.2]{lurie18}, it follows that 
\begin{equation*}
    \pi_i \GW^{T,[n]}(X, \mc{L}) \lra \pi_i \GW^{T,[n]}(X, \mc{L})^{\wedge}_{IO}
\end{equation*}
is an isomorphism, as was to be shown.
\end{proof}

Now we prove our main result, Atiyah-Segal completion for the $T$-equivariant Hermitian $K$-theory of $X$.

\begin{theorem} \label{theorem:ascompletionfortrivialactionofsplittorus}
The natural map
\begin{equation*}
    \pi_{i} \GW^{T,[n]}(X)^{\wedge}_{IO} \lra \pi_i \GW^{[n]}_{\Bo}(X)
\end{equation*}
is an isomorphism for all $i, n \in \Z$.
\end{theorem}

\begin{proof}
The proof is by induction on $i$.
By lemma \ref{lemma:ascompletionforwittgroups}, the statement holds for $i \leq -2$. 
Now let $j \in \Z$ and assume the statement holds for all $i < j$.
For each $n \in \Z$, there is a map of long exact sequences
\begin{equation*}
    \begin{tikzcd}[column sep=small]
            \pi_j \K^T(X)^{\wedge}_{IO} \arrow[r] \arrow[d]
                & \pi_j \GW^{T,[n+1]}(X)^{\wedge}_{IO} \arrow[r] \arrow[d]
                    & \pi_{j-1} \GW^{T,[n]}(X)^{\wedge}_{IO} \arrow[r] \arrow[d] 
                        & \pi_{j-1} \K^T(X)^{\wedge}_{IO} \arrow[d]\\
            \pi_j \K_{\Bo}(X) \arrow[r]
                & \pi_j \GW^{[n+1]}_{\Bo}(X) \arrow[r]
                    & \pi_{j-1} \GW^{[n]}_{\Bo}(X) \arrow[r]
                        & \pi_{j-1} \K_{\Bo}(X),
    \end{tikzcd}
\end{equation*}
in which the maps $\pi_j \K^T(X)^{\wedge}_{IO} \ra \pi_{j-1}\K_{\Bo}(X)$ and $\pi_{j-1} \K^T(X)^{\wedge}_{IO} \ra \pi_{j-1} \K_{\Bo}(X)$ are isomorphisms by the derived filtration lemma \ref{lemma:derivedfiltrationlemma} and \cite[theorem 1.4]{tabuada21} (cf. \cite[theorem 1.2]{krishna18}), while the map 
\begin{equation*}
    \pi_{j-1} \GW^{T,[n]}(X)^{\wedge}_{IO} \lra \pi_{j-1} \GW^{[n]}_{\Bo}(X)    
\end{equation*}
is an isomorphism by the induction hypothesis.
By the appropriate four lemma, the map $\pi_{j} \GW^{T,[n]}(X)^{\wedge}_{IO} \ra \pi_{j} \GW^{[n]}_{\Bo}(X)$ is surjective for each $n \in \Z$. 
Now we can apply the dual four lemma to 
\begin{equation*}
    \begin{tikzcd}[column sep=tiny]
        \pi_j \GW^{T,[n]}(X)^{\wedge}_{IO} \arrow[r] \arrow[d]
            & \pi_j \K^T(X)^{\wedge}_{IO} \arrow[r] \arrow[d]
                & \pi_j \GW^{T,[n+1]}(X)^{\wedge}_{IO} \arrow[r] \arrow[d]
                    & \pi_{j-1} \GW^{T,[n]}(X)^{\wedge}_{IO} \arrow[d] \\
        \pi_{j} \GW^{[n]}_{\Bo}(X) \arrow[r]
            & \pi_{j} \K_{\Bo}(X) \arrow[r]
                & \pi_{j} \GW^{[n+1]}_{\Bo}(X) \arrow[r]
                    & \pi_{j-1} \GW^{[n]}_{\Bo}(X)
    \end{tikzcd}
\end{equation*}
to conclude that $\pi_{j} \GW^{T,[n+1]}(X)^{\wedge}_{IO} \ra \pi_{j} \GW^{[n+1]}_{\Bo}(X)$ is injective and therefore an isomorphism for each $n \in \Z$.
The desired result follows by induction.
\end{proof}

\printbibliography

\end{document}